\documentclass[reqno,b5paper]{amsart}
\usepackage{amsmath}
\usepackage{amssymb}
\usepackage{amsthm}
\usepackage{enumerate}
\usepackage[mathscr]{eucal}
\theoremstyle{plain}
\newtheorem{thm}{Theorem}[section]
\newtheorem{cor}[thm]{Corollary}
\newtheorem{lem}[thm]{Lemma}

\newtheorem{note}{Note}[section]

\theoremstyle{definition}
\newtheorem{defn}{Definition}[section]
\newtheorem{rem}{Remark}[section]

\begin{document}
\setcounter{page}{1}
\title{On strong $\lambda$-statistical convergence of sequences in probabilistic metric (PM) spaces }
\author[P. Malik, S. Das]{Prasanta Malik* and Samiran Das*\ }
\newcommand{\acr}{\newline\indent}
\maketitle
\address{{*\,} Department of Mathematics, The University of Burdwan, Golapbag, Burdwan-713104,
West Bengal, India.
                Email: pmjupm@yahoo.co.in, das91samiran@gmail.com \acr
           }

\maketitle
%---------- Abstract ----------
\begin{abstract}
In this paper we study some basic properties of strong
$\lambda$-statistical convergence of sequences in probabilistic
metric (PM) spaces. We also introduce and study the notion of
strong $\lambda$-statistically Cauchyness. Further introducing the
notions of strong $\lambda$-statistical limit point and strong
$\lambda$-statistical cluster point of a sequence in a
probabilistic metric (PM) space we examine their
interrelationship.

\end{abstract}

\author{}
\maketitle { Key words and phrases: probabilistic metric (PM) space, $\lambda$-density, strong $\lambda$-statistical convergence, strong $\lambda$-statistical Cauchyness, strong $\lambda$-statistical limit point, strong $\lambda$-statistical cluster point.} \\

%---------- 1. Section ----------
\section{\textbf{Introduction:}}

The concept of probabilistic metric (PM) space was introduced and
studied by Menger \cite{Me} under the name of ``statistical metric
space" by taking the distance between two points $a$ and $b$ as a
distribution function $F_{ab}$ instead of a non-negative real
number and the value of the function $F_{ab}$ at any $t> 0 $ i.e.
$F_{ab}(t) $ is interpreted as the probability that the distance
between the points $a$ and $b$ is $\leq t$. After Menger, works of
several mathematicians such as Schwiezer and Sklar \cite{Sh, Sh1,
Sh2, Sh3}, Tardiff \cite{Tar}, Thorp \cite{Th} and many others,
developed the theory of probabilistic metric spaces. A through
development of probabilistic metric spaces can be found in the
famous book of Schwiezer and Sklar \cite{Sh4}. Several topologies
are defined on a PM space but strong topology is one of them,
which received most attention to date  and it is the main tool of
our paper.

As a generalization of the usual notion of convergence of
sequences of real numbers, the notion of statistical convergence
was introduced and studied independently by Fast \cite{Fa} and
Schoenberg \cite{Sc} based on the notion of natural density of
subsets of $\mathbb N$, the set of all natural numbers. A subset
$\mathcal{M}$ of $\mathbb N$ is said to have natural density or
asymptotic density $d(\mathcal{M})$ if
$$d(\mathcal{M})=\lim\limits_{n\rightarrow \infty}\frac{\left|\mathcal{M}(n)\right|}{n}$$
where $\mathcal{M}(n)=\left\{j\in \mathcal{M}:j\leq n\right\}$ and
 $\left|\mathcal{M}(n)\right|$ represents the number of elements in
 $\mathcal{M}(n)$.

A sequence $\{x_k\}_{k\in\mathbb{N}}$ of reals is said to be
statistically convergent to $\xi\in\mathbb{R}$ if for every
$\epsilon>0,~d(A(\epsilon))=0$, where
$A(\epsilon)=\{k\in\mathbb{N}:\left|x_k-\xi\right|\geq\epsilon\}$.

After the great works of Salat \cite{Sa} and Fridy \cite{Fr1},
many works have been done in this area of summability theory. More
primary work on this convergence notion can be found from
\cite{Co1, Co2, Co3,Fr2, Fr3, Pe, St}.

The notion of natural density of subsets of $\mathbb{N}$ was
further generalized to the notion of $\lambda$-density by
Mursaleen \cite{M} and with the help of $\lambda$-density he
generalized the notion of statistical convergence of real
sequences to the notion of $\lambda$-statistical convergence. If
$\lambda=\left\{\lambda_n\right\}_{n\in\mathbb N}$ is a
non-decreasing sequence of positive real numbers tending to
$\infty$ such that
$\lambda_1=1,~~~\lambda_{n+1}\leq\lambda_n+1,~~~n\in\mathbb N$,
then any subset $\mathcal{M}$ of $\mathbb{N}$ is said to have
$\lambda$-density $d_\lambda(\mathcal{M})$ if
$$d_\lambda(\mathcal{M})=\lim\limits_{n\rightarrow\infty}\frac{\left|\{k\in I_n: k\in \mathcal{M}\}\right|}{\lambda_n},$$
where $I_n=[n-\lambda_n + 1,n]$. It is clear that if
$A,B\subset\mathbb{N}$ and $d_\lambda(A)=0$, $d_\lambda(B)=0$ then
$d_\lambda(A^c)=1=d_\lambda(B^c)$, $d_\lambda(A\cup B)=0$. Also if
$A,B\subset\mathbb{N}$, $A\subset B$ and $d_\lambda(B)=0$, then
$d_\lambda(A)=0$. The collection of all such sequences $\lambda$
is denoted by $\Delta_\infty$. Throughout the paper $\lambda$
stands for such a sequence.

If a sequence $x=\{x_k\}_{k \in \mathbb{N}}$ satisfies a property
$\mathfrak{P}$ for each $k$ except for a set of $\lambda$-density
zero, then we say that the sequence $x$ satisfies the property
$\mathfrak{P}$ for ``$\lambda$-almost all $k$'' or in short
``$\lambda\mbox{-}a.a.k.$''.

A sequence $x=\{x_k\}_{k \in \mathbb{N}}$ of real numbers is said
to be $\lambda$-statistically convergent or $S_\lambda$-convergent
to $\mathcal{L}\in\mathbb{R}$ if, for every $\epsilon>0$,
$d_\lambda(M(\epsilon))=0$, where
$M(\epsilon)=\{k\in\mathbb{N}:\left|x_k-\mathcal{L}\right|\geq\epsilon\}$.

If $\lambda_n=n, ~\forall n\in\mathbb{N}$, then the notions of
$\lambda$-density and $\lambda$-statistical convergence coincide
with the notions of natural density and statistical convergence
respectively.

Because of immense importance of probabilistic metric space in
mathematics, the notion of strong statistical convergence was
introduced by \c{S}en\c{c}imen et al. \cite{Se} in a PM space and
this was further generalized to the notion of strong
$\lambda$-statistical convergence by Das et al. \cite{Da3}.

Following the line of \c{S}en\c{c}imen et al. \cite{Se} and also
that of Das et al. \cite{Da3} in this paper we study some basic
properties of strong $\lambda$-statistical convergence of
sequences in probabilistic metric spaces not done earlier. We also
introduce and study the notion of strong $\lambda$-statistical
Cauchyness in a probabilistic metric space. Further we introduce
and study the notions of strong $\lambda$-statistical limit points
and strong $\lambda$-statistical cluster points of a sequence in a
probabilistic metric space including their interrelationship.

%------------------------------Section-2 - Basic definitions-------------------------
\section{\textbf{Basic Definitions and Notations}}

In this section we recall some basic definitions and results
related to probabilistic metric spaces (or PM spaces) (see
\cite{Sh, Sh1, Sh2, Sh3, Sh4} for more details).
%---------------------------Definition 2.1------------------------------------------
\begin{defn}
A nondecreasing function $f: [-\infty, \infty] \rightarrow [0,1]$
with $f(-\infty)=0$ and $ f(\infty)=1,$ is called a distribution
function.
\end{defn}
We denote the set of all distribution functions left continuous
over $(-\infty,\infty)$ by $\mathcal{D}$.

We consider a relation $\leq$ on $\mathcal{D}$ defined by $g \leq
f$ if and only if $g(x)\leq f(x)$ for all $x \in
[-\infty,\infty]$. Clearly the relation `$\leq$' is a partial
order on $\mathcal{D}$.
% --------------------------Definition 2.2-------------------------------------------
\begin{defn}
For any $q\in [-\infty, \infty ]$ the unit step at $q$ is denoted
by $\epsilon_{q}$ and is defined to be a function in $\mathcal{D}$
given by
\begin{eqnarray*}
\epsilon_{q}(x)&=&0, ~~ \text{if}~~-\infty \leq x \leq q\\
                  &=&1,~~ \text{if}~~  q < x\leq \infty.
\end{eqnarray*}
\end{defn}
%-----------------------------Definition 2.3------------------------------------------
\begin{defn}
A sequence $\{f_{n}\}_{n\in \mathbb{N}}$ of distribution functions
is said to converge weakly to a distribution function $f$, if the
sequence $\{f_{n}(x)\}_{n\in\mathbb{N}}$ converges to $f(x)$ at
each continuity point $x$ of $f$. We write
$f_{n}\xrightarrow{w}f$.
\end{defn}
% ----------------------------Definition 2.4-------------------------------------------
\begin{defn}
The distance between f and g in $\mathcal{D}$ is denoted by
$d_{L}(f,g)$ and is defined by the infimum of all numbers
$w\in(0,1]$ such that the inequalities
\begin{center}
$~~~~~~~f(x-w)-w\leq g(x)\leq f(x+w)+w$ \\

and $g(x-w)-w\leq f(x)\leq g(x+w)+w$
\end{center}
hold for every $x\in(-\frac{1}{w},\frac{1}{w})$.
\end{defn}
It is known that $d_{L}$ is a metric on $\mathcal{D}$ and for any
sequence $\{f_{n}\}_{n\in\mathbb{N}}$ in $\mathcal{D}$ and
$f\in\mathcal{D}$, we have
\begin{center}
$f_{n}\xrightarrow{w}f$  if and only if $d_{L}(f_{n},f)\rightarrow
0$.
\end{center}
%---------------------------------Definition 2.5--------------------------------
\begin{defn}
A nondecreasing real valued function $h$ defined on $[0,\infty]$
that satisfies $h(0)=0$ and $ h(\infty)=1$ and is left continuous
on $(0,\infty)$ is called a distance distribution function.
\end{defn}
The set of all distance distribution functions is denoted by
$\mathcal{D}^{+}$. The function $d_{L}$ is clearly a metric on
$\mathcal{D}^{+}$. The metric space $(\mathcal{D}^{+},d_{L})$ is
compact and hence complete.
% ---------------------------------Theorem 2.1------------------------------------
\begin{thm}
Let $f\in\mathcal{D}^{+}$ be given . Then for any $t>0$,
$f(t)>1-t$ if and only if $d_{L}(f,\epsilon_{0})<t$.
\end{thm}
%--------------------------------------Definition2.6-------------------------------
\begin{defn}
A triangle function is a binary operation $\tau$ on
$\mathcal{D}^{+}$ which is commutative, nondecreasing, associative
in each place and $\epsilon_{0}$ is the identity.
\end{defn}
%--------------------------------------Definition2.7--------------------------------
\begin{defn}
A probabilistic metric space, briefly PM space, is a triplet
$(X,\mathcal{F},\tau)$, where $X$ is a nonempty set whose elements
are the points of the space, $\mathcal{F}$ is a function from
$X\times X$ into $\mathcal{D}^{+}$, $\tau$ is a triangle function
and the following conditions are satisfied for all $x,y,z\in X$:
    \begin{enumerate}
    \item $\mathcal{F}(x,x)=\epsilon_{0}$
    \item $\mathcal{F}(x,y)\neq\epsilon_{0}$ if $x\neq y $
    \item $\mathcal{F}(x,y)=\mathcal{F}(y,x) $
    \item $ \mathcal{F}(x,z)\geq \tau(\mathcal{F}(x,y),\mathcal{F}(y,z))$.
    \end{enumerate}
From now on we will denote $\mathcal{F}(x,y)$ by
$\mathcal{F}_{xy}$ and its value at $b$ by $\mathcal{F}_{xy}(b)$.
\end{defn}
%-------------------------------------Definition 2.8-------------------------------------------------------
\begin{defn}
Let $(X,\mathcal{F},\tau)$ be a PM space. For $x\in X$ and $r>0$,
the strong $r$-neighborhood of $x$ is denoted by
$\mathcal{N}_{x}(r)$ and is defined by
\begin{center}
$\mathcal{N}_{x}(r)=\{y\in X : \mathcal{F}_{xy}(r)>1-r\} $.
\end{center}
The collection $\mathfrak{N}_{x}=\{\mathcal{N}_{x}(r):r>0 \}$ is
called the strong neighborhood system at $x$ and the union
$\mathfrak{N}=\underset{x\in X}{\bigcup}\mathfrak{N}_{x}$ is
called the strong neighborhood system for $ X $.
\end{defn}
From Theorem 2.1, we can write $\mathcal{N}_{x}(r)=\{y\in X :
d_{L}(\mathcal{F}_{xy},\epsilon_{0})<r\} $. If $\tau$ is
continuous, then the strong neighborhood system $\mathfrak{N}$
determines a Hausdorff topology for $X$. This topology is called
the strong topology for $X$ and members of this topology are
called strongly open sets.

Throughout the paper, in a PM space $(X,\mathcal{F},\tau)$, we
always consider that $\tau$ is continuous and $X$ is endowed with
the strong topology.

In a PM space $(X,\mathcal{F},\tau)$ the strong closure of any
subset $A$ of $X$ is denoted by $k(A)$ and for any nonempty subset
$A$ of $X$ strong closure of $A$ is defined by,
$$k(A)=\{a\in X: ~\text{for any}~ t>0, ~\exists~ b\in A ~\text{such that}~\mathcal{F}_{ab}(t)>1-t\}.$$

\begin{defn}\cite{Du1}
Let $(X,\mathcal{F},\tau)$ be a PM space. Then a subset $H$ of $X$
is called strongly closed if its complement is a strongly open
set.
\end{defn}
%---------------------------Definition 2.11------------------------------------------
\begin{defn}
Let $(X,\mathcal{F},\tau)$ be a PM space and $H\neq \emptyset$ be
a subset of $X$. Then $\xi\in X$ is said to be a strong limit
point of $H$ if for every $t>0$,
$$\mathcal{N}_\xi(t)\cap(H\setminus\{\xi\})\neq\emptyset.$$
The set of all strong limit points of the set $H$ is denoted by
$L_H^\mathcal{F}$.
\end{defn}
%---------------------------Definition 2.12------------------------------------------
\begin{defn}\cite{Du1}
Let $(X,\mathcal{F},\tau)$ be a PM space and $H$ be a subset of
$X$. Let $\mathfrak{Q}$ be a family of strongly open subsets of
$X$ such that $\mathfrak{Q}$ covers $H$. Then $\mathfrak{Q}$ is
said to be a strong open cover for $H$.
\end{defn}
%---------------------------Definition 2.13------------------------------------------
\begin{defn}\cite{Du1}
Let $(X,\mathcal{F},\tau)$ be a PM space and $H$ be a subset of
$X$. Then $H$ is called a strongly compact set if every strong
open cover of $H$ has a finite subcover.
\end{defn}
%---------------------------Definition 2.14------------------------------------------
\begin{defn}\cite{Du1}
Let $(X,\mathcal{F},\tau)$ be a PM space,
$x=\left\{x_{k}\right\}_{k\in\mathbb N}$ be a sequence in $X$.
Then $x$ is said to be strongly bounded if there exists a strongly
compact subset $E$ of $X$ such that $x_k\in E$, $\forall
k\in\mathbb{N}$.
\end{defn}
%---------------------------Definition 2.15------------------------------------------
\begin{defn}\cite{Du1}
Let $(X,\mathcal{F},\tau)$ be a PM space,
$x=\left\{x_{k}\right\}_{k\in\mathbb N}$ be a sequence in $X$.
Then $x$ is said to be strongly statistically bounded if there
exists a strongly compact subset $E$ of $X$ such that
$d(\{k\in\mathbb{N}:x_k\notin E\})=0$.
\end{defn}
%---------------------------Theorem 2.3------------------------------------------
\begin{thm}\cite{Du1}
Let $(X,\mathcal{F},\tau)$ be a PM space and $H$ be a strongly
compact subset of $X$. Then every strongly closed subset of $H$ is
strongly compact.
\end{thm}
%--------------------------------Definition2.9--------------------------------------------------------------
\begin{defn}
Let $(X,\mathcal{F},\tau)$ be a PM space. Then for any $r>0$, the
subset $\mathfrak{V}(r)$ of $ X\times X$ given by
\begin{center}
$\mathfrak{V}(r)=\{ (x,y): \mathcal{F}_{xy} (r)>1-r\} $
\end{center}
is called the strong $r$-vicinity.
\end{defn}
%-------------------- -------------Theorem2.2--------------------------------------------------------
\begin{thm}
Let $(X,\mathcal{F},\tau)$ be a PM space and $\tau $ be
continuous. Then for any $r>0$, there is an $\eta>0$ such that
$\mathfrak{V}(\eta)\circ\mathfrak{V}(\eta)\subset\mathfrak{V}(r)$,
where $\mathfrak{V}(\eta)\circ\mathfrak{V}(\eta)=\{(x,z):$
\mbox{for some} $y$, $(x,y)$ and $(y,z)\in \mathfrak{V}(\eta)\} $.
 \end{thm}
%-------------------------------------Note 2.1---------------------------------------------------------
\begin{note}
From the hypothesis of Theorem 2.2, we can say that for any $r>0$,
there is an $\eta >0$ such that $\mathcal{F}_{ab}(r)>1-r$ whenever
$\mathcal{F}_{ac}(\eta)>1-\eta $ and
$\mathcal{F}_{cb}(\eta)>1-\eta$. Equivalently it can be written
as: for any $r>0$, there is an $\eta>0$ such that
$d_{L}(\mathcal{F}_{ab},\epsilon_{0})<r$ whenever
$d_{L}(\mathcal{F}_{ac},\epsilon_{0})<\eta$ and
$d_{L}(\mathcal{F}_{cb},\epsilon_{0})<\eta$.
 \end{note}
%---------------------------Definition 2.10------------------------------------------

%---------------------------Definition 2.16------------------------------------------
\begin{defn}\cite{Se}
Let $(X,\mathcal{F},\tau)$ be a PM space. A sequence
$x=\{x_k\}_{k\in \mathbb{N}}$ in $X$ is said to be strongly
convergent to $\mathcal{L}\in X$ if for every $t>0$, $\exists$ a
natural number $k_0$ such that
$$x_k\in\mathcal{N}_\mathcal{L}(t),\hspace{1 in} \text{whenever}~ k\geq k_0.$$
\end{defn}
In this case, we write
$\mathcal{F}$-$\lim\limits_{k\rightarrow\infty}x_k=\mathcal{L}$ or
$x_k\stackrel{\mathcal{F}}\longrightarrow \mathcal{L}$.
%---------------------------Definition 2.17------------------------------------------
\begin{defn}\cite{Sh4}
Let $(X,\mathcal{F},\tau)$ be a PM space. A sequence
$x=\{x_k\}_{k\in \mathbb{N}}$ in $X$ is said to be strong Cauchy
if for every $t>0$, $\exists$ a natural number $k_0$ such that
$$(x_j,x_k)\in\mathfrak{U}(t), \hspace{1 in} \text{whenever}~ j,k\geq k_0 .$$
\end{defn}
%---------------------------Definition 2.18------------------------------------------
\begin{defn}\cite{Se}
Let $(X,\mathcal{F},\tau)$ be a PM space. A sequence $x=\{x_k\}_{k
\in \mathbb{N}}$ in $X$ is said to be strongly statistically
convergent to $\xi\in X$ if for any $t>0$
\begin{center}
$d(\{ k\in\mathbb{N}: \mathcal{F}_{x_k\xi}(t)\leq1-t\})=0
~~~~~~~~\text{or}~~~~~~~~ d(\{ k\in\mathbb{N}:
x_k\notin\mathcal{N}_\xi(t)\})=0.$
\end{center}
In this case we write
$st^{\mathcal{F}}$-$\lim\limits_{k\rightarrow \infty}x_k =\xi$.
\end{defn}
%---------------------------Definition 2.19------------------------------------------
\begin{defn}
Let $(X,\mathcal{F},\tau)$ be a PM space. A sequence $x=\{x_k\}_{k
\in \mathbb{N}}$ in $X$ is said to be strong statistically Cauchy
if for any $t>0$, $\exists$ a natural number $N_0=N_0(t)$ such
that
$$d(\{k\in\mathbb{N}:\mathcal{F}_{x_kx_{N_0}}(t)\leq 1-t\})=0.$$
\end{defn}

\section{\textbf{Strong $\lambda$-Statistical Convergence and Strong $\lambda$-Statistical Cauchyness}}

In this section we first study some basic properties of strong
$\lambda$-statistical convergence in a PM space and then
introducing the notion of strong $\lambda$-statistical Cauchyness
we study its relationship with strong $\lambda$-statistical
convergence.
%---------------------------Definition 3.1------------------------------------------
\begin{defn}\cite{Da3}
Let $(X,\mathcal{F},\tau)$ be a PM space,
$x=\{x_k\}_{k\in\mathbb{N}}$ be a sequence in $X$ and
$\lambda\in\Delta_\infty$. Then $x$ is said to be strongly
$\lambda$-statistically convergent to $\mathcal{L}\in X$, if for
every $t>0$,
$$d_\lambda(\{j\in \mathbb{N}: \mathcal{F}_{x_j\mathcal{L}}(t)\leq1-t\})=0.$$
or
$$d_\lambda(\{j\in \mathbb{N}: x_j \notin \mathcal{N}_\mathcal{L}(t)\})=0.$$
In this case we write,
$st_\lambda^{\mathcal{F}}$-$\lim\limits_{k\rightarrow \infty}x_k =
\mathcal{L} $ or simply as
$x_k\xrightarrow{st^\mathcal{F}_{\lambda}}\mathcal{L}$ and
$\mathcal{L}$ is called strong $\lambda$-statistical limit of $x$.
\end{defn}
%---------------------------Remark 3.1------------------------------------------
\begin{rem}
From Theorem 2.1 and Definition 3.1 we see that the following
statements are equivalent:
\begin{enumerate}
    \item $x_k\xrightarrow{st^\mathcal{F}_{\lambda}}\mathcal{L}$
    \item For each $t>0,d_\lambda(\{k\in \mathbb{N}: d_L(\mathcal{F}_{x_k\mathcal{L}},\epsilon_0)\geq t\})=0$
    \item $ st^\mathcal{F}_\lambda\mbox{-}\lim\limits_{k\rightarrow\infty}d_L(\mathcal{F}_{x_k\mathcal{L}},\epsilon_0)=0$.
\end{enumerate}
\end{rem}
%---------------------------Theorem 3.1------------------------------------------
\begin{thm}
Let $(X,\mathcal{F},\tau)$ be a PM space, $x=\{x_k\}_{k
\in\mathbb{N}}$ be a sequence in $X$ and
$\lambda\in\Delta_\infty$. If the sequence $x=\{x_k\}_{k \in
\mathbb{N}}$ is strongly $\lambda$-statistically convergent in
$X$, then the strong $\lambda$-statistical limit of $x$ is unique.
\end{thm}
\begin{proof}
Let the sequence $x=\{x_k\}_{k \in\mathbb{N}}$ be strongly
$\lambda$-statistically convergent in $X$. If possible, let $
st_\lambda^{\mathcal{F}}$-$\lim\limits_{k\rightarrow\infty}x_k =
\xi_1 $ and $st_\lambda^{\mathcal{F}}$-$\lim\limits_{k\rightarrow
\infty}x_k = \xi_2$ with $\xi_1\neq \xi_2$. Since $\xi_1 \neq
\xi_2$, $\mathcal{F}_{\xi_1\xi_2}\neq\epsilon_0$. Then there is a
$t>0$ such that $d_L(\mathcal{F}_{\xi_1\xi_2},\epsilon_0)=t$. We
choose $\gamma>0$ so that
$d_L(\mathcal{F}_{uv},\epsilon_0)<\gamma$ and
$d_L(\mathcal{F}_{vw},\epsilon_0)<\gamma $ imply that
$d_L(\mathcal{F}_{uw},\epsilon_0)<t$. Since
$st_\lambda^{\mathcal{F}}$-$\lim\limits_{k\rightarrow \infty}x_k =
\xi_1 $ and
$st_\lambda^{\mathcal{F}}$-$\lim\limits_{k\rightarrow\infty}x_k =
\xi_2$, so $d_{\lambda}(A_1(\gamma)) = 0$ and
$d_{\lambda}(A_2(\gamma)) = 0$, where
$$A_1(\gamma)=\{k\in\mathbb{N}: \mathcal{F}_{x_k\xi_1}(\gamma)\leq 1- \gamma\}$$ and
$$A_2(\gamma)=\{k\in \mathbb{N}: \mathcal{F}_{x_k\xi_2}(\gamma)\leq 1- \gamma\}.$$
Now let $A_3(\gamma)=A_1(\gamma)\cup A_2(\gamma)$. Then
$d_\lambda(A_3(\gamma))=0$ and this gives
$d_\lambda(A^c_3(\gamma))=1$. Let $k\in A^c_3(\gamma)$. Then
$d_L(\mathcal{F}_{x_k\xi_1},\epsilon_0)<\gamma$ and
$d_L(\mathcal{F}_{\xi_2x_k},\epsilon_0)<\gamma$ and so
$d_L(\mathcal{F}_{\xi_1\xi_2},\epsilon_0)<t$, this gives a
contradiction. Hence strong $\lambda$-statistical limit of a
strongly $\lambda$-statistically convergent sequence in a PM space
is unique.
\end{proof}
%---------------------------Theorem 3.2------------------------------------------
\begin{thm}
Let $(X,\mathcal{F},\tau)$ be a PM space and
$\{x_n\}_{n\in\mathbb{N}}$, $\{y_n\}_{n\in\mathbb{N}}$ be two
sequences in $X$ such that
$x_n\xrightarrow{st^\mathcal{F}_{\lambda}} l \in X $ and
$y_n\xrightarrow{st^\mathcal{F}_{\lambda}} m \in X$. Then
$$st^\mathcal{F}_\lambda\mbox{-}\lim\limits_{n\rightarrow\infty}d_L(\mathcal{F}_{x_ny_n}, \mathcal{F}_{lm})=0.$$
\end{thm}
\begin{proof}
Since $\tau$ is continuous and $X$ is endowed with the strong
topology, so $\mathcal{F}$ is uniformly continuous. So for any
$t>0$ there exists $\eta(t)>0$ such that
$d_L(\mathcal{F}_{lm},\mathcal{F}_{l_1m_1})<t$, whenever
$l_1\in\mathcal{N}_l(\eta)$ and $m_1\in\mathcal{N}_m(\eta)$. Then
by the given condition, for any $t > 0$
$$\{n\in\mathbb{N} :d_L(\mathcal{F}_{x_ny_n}, \mathcal{F}_{lm})\geq t\}\subset\{n\in \mathbb{N}:x_n\notin\mathcal{N}_l(\eta)\}\cup\{n\in\mathbb{N}:y_n\notin\mathcal{N}_m(\eta)\}.$$ This gives
$$\{n\in I_k:d_L(\mathcal{F}_{x_ny_n}, \mathcal{F}_{lm})\geq t\}\subset\{n\in I_k:x_n\notin \mathcal{N}_l(\eta)\}\cup\{n\in I_k:y_n\notin \mathcal{N}_m(\eta)\}.$$ Thus,
\begin{eqnarray*}
&&~~~d_{\lambda}(\{n\in
\mathbb{N}: d_L(\mathcal{F}_{x_ny_n}, \mathcal{F}_{lm})\geq t\})\\
&\leq&~ d_{\lambda}(\{n\in \mathbb{N}: x_n\notin
\mathcal{N}_l(\eta)\}\cup\{n\in \mathbb{N}: y_n\notin
\mathcal{N}_m(\eta)\}).
\end{eqnarray*}
As $ x_n\xrightarrow{st^\mathcal{F}_{\lambda}}l$ and
$y_n\xrightarrow{st^\mathcal{F}_{\lambda}}m$, so right hand side
of the above inequality is zero and so $$d_{\lambda}(\{n\in
\mathbb{N}: d_L(\mathcal{F}_{x_ny_n}, \mathcal{F}_{lm})\geq
t\})=0.$$ Hence
$st^\mathcal{F}_\lambda\mbox{-}\lim\limits_{n\rightarrow\infty}d_L(\mathcal{F}_{x_ny_n},
\mathcal{F}_{lm})=0$.
\end{proof}
%---------------------------Theorem 3.3------------------------------------------
\begin{thm}
Let $(X,\mathcal{F},\tau)$ be a PM space, $x=\{x_k\}_{k
\in\mathbb{N}}$ be a sequence in $X$ and
$\lambda\in\Delta_\infty$. If the sequence $x=\{x_k\}_{k \in
\mathbb{N}}$ is strongly convergent to $\mathcal{L}\in X$, then
$st^\mathcal{F}_\lambda\mbox{-}\lim\limits_{k\rightarrow\infty}x_k=\mathcal{L}$.
\end{thm}
\begin{proof}
The proof is trivial, so is omitted.
\end{proof}
%---------------------------Theorem 3.4------------------------------------------
\begin{thm}
Let $(X,\mathcal{F},\tau)$ be a PM space,
$x=\{x_k\}_{k\in\mathbb{N}}$ be a sequence in $X$ and
$\lambda\in\Delta_{\infty}$. Then
$st_\lambda^\mathcal{F}$-$\lim\limits_{k\rightarrow\infty}x_k=\mathcal{L}$
if and only if there is a subset $G=\{q_1<q_2<...\}$ of
$\mathbb{N}$ such that $d_\lambda(G)=1$ and
$\mathcal{F}$-$\lim\limits_{n\rightarrow\infty}x_{q_n}=\mathcal{L}$.
\end{thm}
\begin{proof}
Let us assume that
$st^\mathcal{F}_\lambda\mbox{-}\lim\limits_{k\rightarrow\infty}x_k=\mathcal{L}$.
Then for each $t\in\mathbb{N}$, let
$$E_t=\{q\in \mathbb{N}:d_L(\mathcal{F}_{x_q\mathcal{L}},\epsilon_0)\geq\frac{1}{t}\}$$
and
$$G_t=\{q\in \mathbb{N}:d_L(\mathcal{F}_{x_q\mathcal{L}},\epsilon_0)<\frac{1}{t}\}.$$

Then from Remark 3.1, we get $d_\lambda(E_t)=0$. Also by
construction of $G_t$ for each $t\in \mathbb{N}$ we have
$G_1\supset G_2\supset G_3\supset...\supset G_m\supset
G_{m+1}\supset...$ with $d_\lambda(G_t)=1$ for each
$t\in\mathbb{N}$.

Let $u_1\in G_1$. As $d_\lambda(G_2)=1$, so $\exists ~u_2\in G_2$
with $u_2>u_1$ such that for each $n\geq u_2$,
$\frac{\left|G_2(n)\right|}{\lambda_n}>\frac{1}{2}$ where
$G_t(n)=\{k\in I_n: k\in G_t\}$ and $\left|G_t(n)\right|$ is the
number of element in the set $G_t(n)$ for each $t\in \mathbb{N}$.

Again, as $d_\lambda(G_3)=1$, so $\exists~u_3\in G_3$ with
$u_3>u_2$ such that for each $n\geq u_3$,
$\frac{\left|G_3(n)\right|}{\lambda_n}>\frac{2}{3}$.

Thus we set a strictly increasing sequence
$\{u_t\}_{t\in\mathbb{N}}$ of positive integers such that $u_t\in
G_t$ for each $t\in\mathbb{N}$ and
$$\frac{\left|G_t(n)\right|}{\lambda_n}>\frac{t-1}{t} \hspace{1in} \text{for each}~n\geq u_t, t\in\mathbb{N}.$$

We now define the set $G$ as follows
$$G=\biggl\{k\in\mathbb{N}:k\in[1,u_1]\biggr\}\bigcup\biggl\{\bigcup\limits_{t\in\mathbb{N}}\{k\in\mathbb{N}:k\in[u_t,u_{t+1}]~\text{and}~k\in G_t\}\biggr\}.$$

Then, for each $n$, $u_t\leq n<u_{t+1}$, we have

$$\frac{\left|G(n)\right|}{\lambda_n}\geq\frac{\left|G_t(n)\right|}{\lambda_n}>\frac{t-1}{t}.$$

Therefore $d_\lambda(G)=1$.

Let $\eta>0$. We choose $l\in\mathbb{N}$ such that
$\frac{1}{l}<\eta$. Let $n\geq u_l$, $n\in G$. Then $\exists$ a
natural number $r\geq l$ such that $u_r\leq n<u_{r+1}$. Then by
the construction of $G$, $n\in G_r$. So,
$$d_L(\mathcal{F}_{x_n\mathcal{L}},\epsilon_0)<\frac{1}{r}\leq\frac{1}{l}<\eta.$$

Thus $d_L(\mathcal{F}_{x_n\mathcal{L}},\epsilon_0)<\eta$ for each
$n\in G,~n\geq u_l$. Hence
$\mathcal{F}\mbox{-}\lim\limits_{\stackrel{\stackrel{k\rightarrow\infty}{k\in
G}}~}x_k=\mathcal{L}$. Writing $G=\{q_1<q_2<...\}$ we have
$d_\lambda(G)=1$ and
$\mathcal{F}\mbox{-}\lim\limits_{n\rightarrow\infty}x_{q_n}=\mathcal{L}$.

Conversely, let there exists a subset $G=\{q_1<q_2<...\}$ of
$\mathbb{N}$ such that $d_\lambda(G)=1$ and
$\mathcal{F}\mbox{-}\lim\limits_{n\rightarrow\infty}x_{q_n}=\mathcal{L}(\in
X)$. Then for each $t>0$, there is an $N_0\in\mathbb{N}$ so that
$$ \mathcal{F}_{x_{q_n}\mathcal{L}}(t)>1-t, \hspace{1in} \forall~ n\geq N_0,$$
i.e.,
$$d_L(\mathcal{F}_{x_{q_n}\mathcal{L}},\epsilon_0)<t, \hspace{1in} \forall~ n\geq N_0.$$
Let $r>0$ be a real number and $E_r=\{n\in
\mathbb{N}:d_L(\mathcal{F}_{x_{q_n}\mathcal{L}},\epsilon_0)\geq
r\}$. Then $E_r\subset\mathbb{N} \setminus \{q_{ _{N_0+1}},q_{
_{N_0+2}},...\}$. Now $d_\lambda(\mathbb{N} \setminus \{q_{
_{N_0+1}}, q_{ _{N_0+2}},...\}) = 0$ and so $d_\lambda(E_r)=0$.
Therefore,
$st^\mathcal{F}_\lambda\mbox{-}\lim\limits_{k\rightarrow\infty}x_k=\mathcal{L}$.
\end{proof}
%---------------------------Theorem 3.5------------------------------------------
\begin{thm}
Let $(X,\mathcal{F},\tau)$ be a PM space, $x=\{x_k\}_{k \in
\mathbb{N}}$ be a sequence in $X$ and $\lambda\in\Delta_\infty$.
Then $x_k\xrightarrow{st^\mathcal{F}_{\lambda}}\mathcal{L}$ if and
only if there exists a sequence $\{g_k\}_{k\in\mathbb{N}}$ such
that $x_k=g_k$ for $\lambda\mbox{-}a.a.k.$ and
$g_k\xrightarrow{\mathcal{F}}\mathcal{L}$.
\end{thm}
\begin{proof}
Follows from Theorem 3.4.
\end{proof}
%---------------------------Definition 3.3------------------------------------------
\begin{defn}
Let $(X,\rho)$ be a metric space and $x=\{x_k\}_{k\in\mathbb{N}}$
be a sequence in $X$. Then $x$ is said to be
$\lambda$-statistically Cauchy in $X$ if for every $\eta>0$, there
exists a natural number $N_0$ such that
$$d_\lambda(\{k\in\mathbb{N}:\rho(x_k,x_{N_0})\geq\eta\})=0.$$
\end{defn}

Now as a consequence of the proposition 4., of \cite{De1}, we get
the following lemma.
%---------------------------Lemma 3.9------------------------------------------
\begin{lem}
Let $(X,\rho)$ be a metric space and $x=\{x_k\}_{k\in\mathbb{N}}$
be a sequence in $X$. Then the following statements are
equivalent:
\begin{enumerate}
    \item $x$ is a $\lambda$-statistically Cauchy sequence.
    \item  For all $\eta>0$, there is a set $G\subset\mathbb{N}$ such that $d_\lambda(G)=0$ and $\rho(x_m,x_n)<\eta$ for all $m,n\notin G$.
    \item For every $\eta>0$, $d_\lambda(\{j\in\mathbb{N}:d_\lambda(D_j)\neq 0\})=0$, where $D_j(\eta)=\{k\in\mathbb{N}:\rho(x_k,x_j)\geq \eta\}$, $j\in\mathbb{N}$.
\end{enumerate}
\end{lem}
\begin{proof}
The proof is trivial, so is omitted.
\end{proof}
%---------------------------Definition 3.2------------------------------------------
\begin{defn}
Let $(X,\mathcal{F},\tau)$ be a PM space, $x=\{x_k\}_{k \in
\mathbb{N}}$ be a sequence in $X$ and $\lambda\in\Delta_\infty$.
Then $x$ is said to be strong $\lambda$-statistically Cauchy
sequence if for every $t>0$, $\exists$ a natural number $N_0$
depending on $t$ such that
$$d_\lambda(\{k\in \mathbb{N}: \mathcal{F}_{x_kx_{N_0}}(t)\leq 1-t\})=0.$$
\end{defn}
%---------------------------Theorem 3.6------------------------------------------
\begin{thm}
Let $(X,\mathcal{F},\tau)$ be a PM space, $x=\{x_k\}_{k \in
\mathbb{N}}$ be a sequence in $X$ and $\lambda\in\Delta_\infty$.
If $x$ is strongly $\lambda$-statistically convergent, then $x$ is
strong $\lambda$-statistically Cauchy.
\end{thm}
\begin{proof}
The proof is trivial, so is omitted.
\end{proof}
%---------------------------Theorem 3.7------------------------------------------
\begin{thm}
Let $(X,\mathcal{F},\tau)$ be a PM space, $x=\{x_k\}_{k \in
\mathbb{N}}$ be a sequence in $X$ and $\lambda\in\Delta_\infty$.
If the sequence $x=\{x_k\}_{k \in \mathbb{N}}$ is strong
$\lambda$-statistically Cauchy, then for each $t>0$, there is a
set $H_t\subset \mathbb{N}$ with $d_\lambda(H_t)=0$ such that
$\mathcal{F}_{x_kx_j}(t)>1-t$ for any $k,j\notin H_t$.
\end{thm}
\begin{proof}
Let $x=\{x_k\}_{k \in \mathbb{N}}$ be strong
$\lambda$-statistically Cauchy. Let $t>0$. Then by Note 2.1, there
is a $\gamma=\gamma(t)>0$ such that,
$$\mathcal{F}_{\mathcal{L}r}(t)>1-t ~\text{whenever}~ \mathcal{F}_{\mathcal{L}j}(\gamma)>1-\gamma ~\text{and}~ \mathcal{F}_{jr}(\gamma)>1-\gamma.$$

As the sequence $x=\{x_k\}_{k \in \mathbb{N}}$ is strong
$\lambda$-statistically Cauchy, so there is an
$N_0=N_0(\gamma)\in\mathbb{N}$ such that
$$d_\lambda(\{k\in \mathbb{N}:\mathcal{F}_{x_kx_{N_0}}(\gamma)\leq 1-\gamma\})=0.$$
Let $H_t=\{k\in \mathbb{N}:\mathcal{F}_{x_kx_{N_0}}(\gamma)\leq
1-\gamma\}$. Then $d_\lambda(H_t)=0$ and
$\mathcal{F}_{x_kx_{N_0}}(\gamma)>1-\gamma$ and
$\mathcal{F}_{x_jx_{N_0}}(\gamma)>1-\gamma$ for $k,j\notin H_t$.
Hence for every $t>0$, there is a set $H_t\subset\mathbb{N}$ with
$d_\lambda(H_t)=0$ such that $\mathcal{F}_{x_kx_j}(t)>1-t$ for
every $k,j\notin H_t$.
\end{proof}
%---------------------------Corollary 3.9------------------------------------------
\begin{cor}
Let $(X,\mathcal{F},\tau)$ be a PM space, $x=\{x_k\}_{k \in
\mathbb{N}}$ be a sequence in $X$ and $\lambda\in\Delta_\infty$.
If the sequence $x=\{x_k\}_{k \in \mathbb{N}}$ is strong
$\lambda$-statistically Cauchy, then for each $t>0$, there is a
set $G_t\subset \mathbb{N}$ with $d_\lambda(G_t)=1$ such that
$\mathcal{F}_{x_kx_j}(t)>1-t$ for any $k,j\in G_t$.
\end{cor}
%---------------------------Theorem 3.10------------------------------------------
\begin{thm}
Let $(X,\mathcal{F},\tau)$ be a PM space, $x=\{x_k\}_{k \in
\mathbb{N}}$, $g=\{g_k\}_{k\in\mathbb{N}}$ be two strong
$\lambda$-statistically Cauchy sequences in $X$ and
$\lambda\in\Delta_\infty$. Then
$\{\mathcal{F}_{{x_k}{g_k}}\}_{k\in\mathbb{N}}$ is a
$\lambda$-statistically Cauchy sequence in $(\mathcal{D}^+,d_L)$.
\end{thm}
\begin{proof}
As $x=\{x_k\}_{k \in \mathbb{N}}$ and $g=\{g_k\}_{k\in\mathbb{N}}$
are strong $\lambda$-statistically Cauchy sequences, so by
corollary 3.9, for every $\gamma>0$ there are $U_\gamma,
V_\gamma\subset\mathbb{N}$ with
$d_\lambda(U_\gamma)=d_\lambda(V_\gamma)=1$ so that
$\mathcal{F}_{x_qx_j}(\gamma)>1-\gamma$ holds for any $q,j\in
U_\gamma$ and $\mathcal{F}_{g_sg_t}(\gamma)>1-\gamma$ holds for
any $s,t\in V_\gamma$. Let $W_\gamma=U_\gamma\cap V_\gamma$. Then
$d_\lambda(W_\gamma)=1$. So, for every $\gamma>0$, there is a set
$W_\gamma\subset\mathbb{N}$ with $d_\lambda(W_\gamma)=1$ so that
$\mathcal{F}_{x_px_r}(\gamma)>1-\gamma$ and
$\mathcal{F}_{g_pg_r}(\gamma)>1-\gamma$ for any $p,r\in W_\gamma$.
Now let $t>0$. Then there exists a $\gamma(t)$ and hence a set
$W_\gamma=W_t\subset\mathbb{N}$ with $d_\lambda(W_t)=1$ so that
$d_L(\mathcal{F}_{x_pg_p},\mathcal{F}_{x_rg_r})<t$ for any $p,r\in
W_t$, as $\mathcal{F}$ is uniformly continuous. Then the result
follows from Lemma 3.6.
\end{proof}

%---------------------------Section-4------------------------------------------
\section{\textbf{Strong $\lambda$-Statistical Limit Points and Strong $\lambda$-Statistical Cluster Points}}

In this section we introduce the notions of strong
$\lambda$-statistical limit points and strong
$\lambda$-statistical cluster points of a sequence in a PM space
and study some of their properties.
%---------------------------Definition 4.1------------------------------------------
\begin{defn}
Let $(X,\mathcal{F},\tau)$ be a PM space, $x=\{x_k\}_{k
\in\mathbb{N}}$ be a sequence in $X$ and
$\lambda\in\Delta_\infty$. If $\{x_{q_k}\}_{k\in\mathbb{N}}$ is a
subsequence of the sequence $x$ and
$\mathcal{B}=\{q_k:k\in\mathbb{N}\}$ then we denote
$\{x_{q_k}\}_{k\in\mathbb{N}}$ by $\{x\}_\mathcal{B}$. Now if
$d_\lambda(\mathcal{B}) = 0$, then $\{x\}_\mathcal{B}$ is called a
$\lambda$-thin subsequence of $x$. On the other hand,
$\{x\}_\mathcal{B}$ is called a $\lambda$-nonthin subsequence of
$x$, if $\mathcal{B}$ does not have $\lambda$-density zero i.e.,
if either $d_\lambda(\mathcal{B})$ is a positive number or
$\mathcal{B}$ fails to have $\lambda$-density.
\end{defn}
%---------------------------Definition 4.2------------------------------------------
\begin{defn}\cite{Se}
Let $(X,\mathcal{F},\tau)$ be a PM space, $x=\{x_k\}_{k
\in\mathbb{N}}$ be a sequence in $X$ and
$\lambda\in\Delta_\infty$. An element $l\in X$ is said to be a
strong limit point of the sequence $x$, if there exists a
subsequence of $x$ that strongly converges to $l$.
\end{defn}
To denote the set of all strong limit points of any sequence $x$
in a PM space $(X,\mathcal{F},\tau)$ we use the notation
$L_x^\mathcal{F}$.
%---------------------------Definition 4.3------------------------------------------
\begin{defn}
Let $(X,\mathcal{F},\tau)$ be a PM space, $x=\{x_k\}_{k
\in\mathbb{N}}$ be a sequence in $X$ and
$\lambda\in\Delta_\infty$. An element $\mathcal{L}\in X$ is said
to be a strong $\lambda$-statistical limit point of the sequence
$x=\{x_k\}_{k\in\mathbb N}$, if there exists a $\lambda$-nonthin
subsequence of $x$ that strongly converges to $\mathcal{L}$.
\end{defn}
To denote the set of all strong $\lambda$-statistical limit points
of any sequence $x$ in a PM space $(X,\mathcal{F},\tau)$ we use
the notation $\Lambda_x^{st}(\lambda)^\mathcal{F}_{s}$.
%---------------------------Definition 4.4------------------------------------------
\begin{defn}
Let $(X,\mathcal{F},\tau)$ be a PM space, $x=\{x_k\}_{k
\in\mathbb{N}}$ be a sequence in $X$ and
$\lambda\in\Delta_\infty$. An element $\mathcal{Y}\in X$ is said
to be a strong $\lambda$-statistical cluster point of the sequence
$x=\{x_k\}_{k\in\mathbb N}$, if for every $t>0$, the set $\{k
\in\mathbb N : \mathcal{F}_{x_k\mathcal{Y}}(t)>1-t \}$ does not
have $\lambda$-density zero.
\end{defn}
To denote the set of all strong $\lambda$-statistical cluster
points of any sequence $x$ in a PM space $(X,\mathcal{F},\tau)$ we
use the notation $\Gamma_x^{st}(\lambda)^\mathcal{F}_{s}$.
%-------------------------------Note 4.1------------------------------------------
\begin{note}
If we choose $\lambda_n=n$ for all $n\in\mathbb{N}$, then strong
$\lambda$-statistical limit points and strong
$\lambda$-statistical cluster points coincide with strong
statistical limit points and strong statistical cluster points of
a sequence respectively in a PM space as introduced in \cite{Se}.
\end{note}
%---------------------------Theorem 4.1------------------------------------------
\begin{thm}
Let $(X,\mathcal{F},\tau)$ be a PM space, $x=\{x_k\}_{k
\in\mathbb{N}}$ be a sequence in $X$ and
$\lambda\in\Delta_\infty$. Then
$\Lambda_x^{st}(\lambda)^\mathcal{F}_{s} \subset
\Gamma_x^{st}(\lambda)^\mathcal{F}_{s}\subset L_x^\mathcal{F}$.
\end{thm}
\begin{proof}
Let $\xi \in \Lambda_x^{st}(\lambda)^\mathcal{F}_{s}$. Then we get
a subsequence $\{x_{k_n}\}_{n\in \mathbb N}$ of the  sequence $x$
such that $\mathcal{F}\mbox{-}\lim\limits_{n\rightarrow
\infty}x_{k_n}=\xi$ and $d_\lambda(\mathcal{M})\neq 0$, where
$\mathcal{M} = \{k_n\in\mathbb {N}: n\in\mathbb {N}\}$. Suppose
$t>0$ be arbitrary. Since
$\mathcal{F}\mbox{-}\lim\limits_{n\rightarrow \infty}x_{k_n}=\xi$,
so $\exists~ p_0\in\mathbb N$ such that
$\mathcal{F}_{x_{k_n}\xi}(t)>1-t~\text{whenever}~ n\geq p_0$. Let
$\mathcal{B}=\{k_1,k_2,...,k_{p_0-1}\}$. Then,
\begin{eqnarray*}
&~&
\{k\in\mathbb{N}:\mathcal{F}_{x_{k}\xi}(t)>1-t\}\supset \{k_n\in\mathbb{N}: n \in \mathbb{N}\}\backslash \mathcal{B}\\
&\Rightarrow & \mathcal{M}=\{k_n\in\mathbb{N}:n\in \mathbb
N\}\subset \{k:\mathcal{F}_{x_{k}\xi}(t)>1-t\}\cup\mathcal{B}.
\end{eqnarray*}
Now if $d_\lambda(\{k\in
\mathbb{N}:\mathcal{F}_{x_{k_n}\xi}(t)>1-t\})=0$, then we get
$d_\lambda(\mathcal{M})=0$, a contradiction. Hence $\xi$ is a
strong $\lambda$-statistical cluster point of $x$. Since $\xi \in
\Lambda_x^{st}(\lambda)^\mathcal{F}_{s}$ is arbitrary, so
$\Lambda_x^{st}(\lambda)^\mathcal{F}_{s}\subset\Gamma_x^{st}(\lambda)^\mathcal{F}_{s}$.

Now let $\alpha\in\Gamma_x^{st}(\lambda)^\mathcal{F}_{s}$. Then
$\lambda$-density of the set
$$\{k\in\mathbb{N}:\mathcal{F}_{x_k\alpha}(t)>1-t\}$$ is not zero, for every $t>0$.
So there exists a subsequence $\{x\}_\mathcal{K}$ of $x$ that
strongly converges to $\alpha$. So, $\alpha\in L^\mathcal{F}_x$.

Therefore, $\Gamma_x^{st}(\lambda)^\mathcal{F}_{s}\subset
L^\mathcal{F}_x$.
\end{proof}
%---------------------------Theorem 4.2------------------------------------------
\begin{thm}
Let $(X,\mathcal{F},\tau)$ be a PM space, $x=\{x_k\}_{k
\in\mathbb{N}}$ be a sequence in $X$ and
$\lambda\in\Delta_\infty$. If
$st^\mathcal{F}_\lambda\mbox{-}\lim\limits_{k\rightarrow
\infty}x_k = \alpha$, then
$\Lambda_x^{st}(\lambda)^\mathcal{F}_{s}=\Gamma_x^{st}(\lambda)^\mathcal{F}_{s}=\{\alpha\}$.
\end{thm}
\begin{proof}
Let $st^\mathcal{F}_\lambda\mbox{-}\lim\limits_{k\rightarrow
\infty}x_k = \alpha$. So for every $t>0$,
$d_\lambda(\{k\in\mathbb{N}:\mathcal{F}_{x_k\alpha}(t)>1-t\})=1$.
Therefore, $\alpha\in\Gamma_x^{st}(\lambda)^\mathcal{F}_{s}$. Now
assume that there exists at least one
$\beta\in\Gamma_x^{st}(\lambda)^\mathcal{F}_{s}$ such that
$\alpha\neq\beta$. Then $\mathcal{F}_{\alpha\beta}\neq\epsilon_0$.
Then there is a $t_1>0$ such that
$d_L(\mathcal{F}_{\alpha\beta},\epsilon_0)=t_1$. Then there exists
$t>0$ such that $d_L(\mathcal{F}_{uv},\epsilon_0)<t$ and
$d_L(\mathcal{F}_{vw},\epsilon_0)<t $ imply that
$d_L(\mathcal{F}_{uw},\epsilon_0)<t_1$. Now since
$\alpha,\beta\in\Gamma_x^{st}(\lambda)^\mathcal{F}_{s}$, for that
$t>0$, $d_\lambda(G)\neq 0$ and $d_\lambda(H)\neq 0$, where
$G=\{k\in\mathbb{N}: \mathcal{F}_{x_k\alpha}(t)> 1-t\}$ and
$H=\{k\in\mathbb{N}: \mathcal{F}_{x_k\beta}(t)> 1-t\}$. As,
$\alpha\neq\beta$, so $G\cap H=\emptyset$ and so $H\subset G^c$.
Since $st^\mathcal{F}_\lambda\mbox{-}\lim\limits_{k\rightarrow
\infty}x_k = \alpha$ so $d_\lambda(G^c)=0$. Then $d_\lambda(H)=0$,
which is a contradiction.

Therefore, $\Gamma_x^{st}(\lambda)^\mathcal{F}_{s}=\{\alpha\}$.

As $st^\mathcal{F}_\lambda\mbox{-}\lim\limits_{k\rightarrow
\infty}x_k = \alpha$, so from Theorem 3.5, we have
$\alpha\in\Lambda_x^{st}(\lambda)^\mathcal{F}_{s}$. Now by Theorem
4.1, we get
$\Lambda_x^{st}(\lambda)^\mathcal{F}_{s}=\Gamma_x^{st}(\lambda)^\mathcal{F}_{s}=\{\alpha\}$.
\end{proof}
%---------------------------Theorem 4.3------------------------------------------
\begin{thm}
Let $(X,\mathcal{F},\tau)$ be a PM space,
$\lambda\in\Delta_\infty$ and $x=\{x_k\}_{k\in \mathbb N}$,
$y=\{y_k\}_{k\in\mathbb N}$ be two sequences in $X$ such that
$d_\lambda(\{k\in \mathbb{N} : x_k \neq y_k \})=0$. Then
$\Lambda_x^{st}(\lambda)^\mathcal{F}_{s}=\Lambda_y^{st}(\lambda)^\mathcal{F}_{s}$
and $\Gamma_x^{st}(\lambda)^\mathcal{F}_{s}=
\Gamma_y^{st}(\lambda)^\mathcal{F}_{s}$.
\end{thm}
\begin{proof}
Let $\xi \in \Gamma_x^{st}(\lambda)^\mathcal{F}_{s}$ and
$\epsilon>0$ be given. Let $\mathcal{B}=\{k\in\mathbb N: x_k =
y_k\}$. Since $d_\lambda(\mathcal{B})=1$, so
$d_\lambda(\{k\in\mathbb N:\mathcal{F}_{x_k\xi}(t)>1-t\}\cap
\mathcal{B})$ is not zero. This gives $d_\lambda(\{k\in\mathbb
N:\mathcal{F}_{y_k\xi}(t)>1-t\})$ is not zero and so $\xi \in
\Gamma_y^{st}(\lambda)^\mathcal{F}_{s}$. Since $\xi \in
\Gamma_x^{st}(\lambda)^\mathcal{F}_{s}$ is arbitrary, so
$\Gamma_x^{st}(\lambda)^\mathcal{F}_{s}\subset
\Gamma_y^{st}(\lambda)^\mathcal{F}_{s}$. Similarly, we get $
\Gamma _x^{st}(\lambda)^\mathcal{F}_{s} \supset
\Gamma_y^{st}(\lambda)^\mathcal{F}_{s}$. Hence
$\Gamma_x^{st}(\lambda)^\mathcal{F}_{s}=
\Gamma_y^{st}(\lambda)^\mathcal{F}_{s}$.

Now let $\eta\in \Lambda_y^{st}(\lambda)^\mathcal{F}_{s}$. Then
$y$ has a $\lambda$-nonthin subsequence $\{y_{k_n}\}_{n\in \mathbb
N}$ that strongly converges to $\eta$. Let $\mathcal{Z}=\{k_n \in
\mathbb N : y_{k_n} =  x_{k_n}\}$. Since $d_\lambda(\{k_n\in
\mathbb N : y_{k_n}\neq x_{k_n}\})=0$ and $\{y_{k_n}\}_{n\in
\mathbb N}$ is a $\lambda$-nonthin subsequence of $y$ so
$d_\lambda(\mathcal{Z}) \neq 0$. Using the set $\mathcal{Z}$ we
get a $\lambda$-nonthin subsequence $\{x\}_{\mathcal{Z}}$ of the
sequence $x$ that strongly converges to $\eta$. Thus $\eta \in
\Lambda_x^{st}(\lambda)^\mathcal{F}_{s}$. Since $\eta \in \Lambda
_y^{st}(\lambda)^\mathcal{F}_{s}$ is arbitrary, so
$\Lambda_y^{st}(\lambda)^\mathcal{F}_{s}\subset
\Lambda_x^{st}(\lambda)^\mathcal{F}_{s}$. By similar argument, we
get $\Lambda_x^{st}(\lambda)^\mathcal{F}_{s}\subset \Lambda
_y^{st}(\lambda)^\mathcal{F}_{s}$. Hence
$\Lambda_x^{st}(\lambda)^\mathcal{F}_{s}=
\Lambda_y^{st}(\lambda)^\mathcal{F}_{s}$.
\end{proof}
%---------------------------Theorem 4.4------------------------------------------
\begin{thm}
Let $(X,\mathcal{F},\tau)$ be a PM space,
$x=\left\{x_{k}\right\}_{k\in\mathbb N}$ be a sequence in $X$ and
$\lambda\in\Delta_\infty$. Then the set
$\Gamma_x^{st}(\lambda)^\mathcal{F}_{s}$ is a strongly closed set.
\end{thm}
\begin{proof}
The proof is trivial, so is omitted.
\end{proof}
%---------------------------Theorem 4.5------------------------------------------
\begin{thm}
Let $(X,\mathcal{F},\tau)$ be a PM space,
$x=\left\{x_{k}\right\}_{k\in\mathbb N}$ be a sequence in $X$ and
$\lambda\in\Delta_\infty$. Let $C$ be a strongly compact subset of
$X$ such that
$C\cap\Gamma_x^{st}(\lambda)^\mathcal{F}_{s}=\emptyset$. Then
$d_\lambda(G)=0$, where $G=\{k\in\mathbb{N}:x_k\in C\}$.
\end{thm}
\begin{proof}
As $C\cap\Gamma_x^{st}(\lambda)^\mathcal{F}_{s}=\emptyset$, so for
all $\beta\in C$, there exists a real number $t=t(\beta)>0$ so
that
$d_\lambda(\{k\in\mathbb{N}:\mathcal{F}_{x_k\beta}(t)>1-t\})=0$.
Let $B_\beta(t)=\{a\in X:\mathcal{F}_{a\beta}(t)>1-t\}$. Then the
family of strongly open sets $\mathfrak{Q}=\{B_\beta(t):\beta\in
C\}$ forms a strong open cover of $C$. As $C$ is a strongly
compact set, so there exists a finite subcover
$\{B_{\beta_1}(t_1),B_{\beta_2}(t_2),...,B_{\beta_m}(t_m)\}$ of
the strong open cover $\mathfrak{Q}$. Then
$C\subset\bigcup\limits_{j=1}^mB_{\beta_j}(t_j)$ and also for each
$j=1,2,...,m$ we have
$d_\lambda(\{k\in\mathbb{N}:\mathcal{F}_{x_k\beta_j}(t_j)>1-t_j\})=0$.
So,
$$\left|\{k\in\mathbb{N}:x_k\in C\right\}|\leq\sum\limits_{j=1}^m\left|\{k\in\mathbb{N}:\mathcal{F}_{x_k\beta_j}(t_j)>1-t_j\}\right|.$$ Then,
$$\lim\limits_{n\rightarrow\infty}\frac{1}{\lambda_n}\left|\{k\in
I_n: x_k\in C\}\right|
\leq\lim\limits_{n\rightarrow\infty}\frac{1}{\lambda_n}\sum\limits_{j=1}^m\left|\{k\in
I_n:\mathcal{F}_{x_k\beta_j}(t_j)>1-t_j\}\right|=0.$$ This gives
$d_\lambda(G)=0$, where $G=\{k\in\mathbb{N}:x_k\in C\}$.
\end{proof}
%---------------------------Theorem 4.6------------------------------------------
\begin{thm}
Let $(X,\mathcal{F},\tau)$ be a PM space and
$x=\left\{x_{k}\right\}_{k\in\mathbb N}$ be a sequence in $X$. If
$x$ has a strongly bounded $\lambda$-nonthin subsequence, then the
set $\Gamma_x^{st}(\lambda)^\mathcal{F}_{s}$ is nonempty and
strongly closed.
\end{thm}
\begin{proof}
Let $\{x\}_\mathcal{B}$ be a strongly bounded $\lambda$-nonthin
subsequence of $x$. So $d_\lambda(\mathcal{B})\neq 0$ and there
exists a strongly compact subset $C$ of $X$ such that $x_k\in C$
for all $k\in \mathcal{B}$. If
$\Gamma_x^{st}(\lambda)^\mathcal{F}_{s}=\emptyset$ then
$C\cap\Gamma_x^{st}(\lambda)^\mathcal{F}_{s}=\emptyset$ and then
by Theorem 4.5, we get $d_\lambda(G)=0$, where
$G=\{k\in\mathbb{N}:x_k\in C\}$. But $\left|\{k\in
I_n:k\in\mathcal{B}\}\right|\leq \left|\{k\in I_n:x_k\in
C\}\right|$, which gives $d_\lambda(\mathcal{B})=0$, which
contradicts our assumption. Hence
$\Gamma_x^{st}(\lambda)^\mathcal{F}_{s}$ is nonempty and also by
Theorem 4.4, $\Gamma_x^{st}(\lambda)^\mathcal{F}_{s}$ is strongly
closed.
\end{proof}
%---------------------------Definition 4.3------------------------------------------
\begin{defn}
Let $(X,\mathcal{F},\tau)$ be a PM space and
$x=\left\{x_{k}\right\}_{k\in\mathbb N}$ be a sequence in $X$.
Then $x$ is said to be strongly $\lambda$-statistically bounded if
there exists a strongly compact subset $C$ of $X$ such that
$d_\lambda(\{k\in\mathbb{N}:x_k\notin C\})=0$.
\end{defn}
%---------------------------Theorem 4.7------------------------------------------
\begin{thm}
Let $(X,\mathcal{F},\tau)$ be a PM space and
$x=\left\{x_{k}\right\}_{k\in\mathbb N}$ be a sequence in $X$. If
$x$ is strongly $\lambda$-statistically bounded then the set
$\Gamma_x^{st}(\lambda)^\mathcal{F}_{s}$ is nonempty and strongly
compact.
\end{thm}
\begin{proof}
Let $x$ be strongly $\lambda$-statistically bounded. Let $C$ be a
strongly compact set with $d_\lambda(E)=0$, where
$E=\{k\in\mathbb{N}:x_k\notin C\}$. Then $d_\lambda(E^c)=1\neq 0$
and so $C$ contains a $\lambda$- nonthin subsequence of $x$. So,
by Theorem 4.6, $\Gamma_x^{st}(\lambda)^\mathcal{F}_{s}$ is
nonempty and strongly closed. We now prove that
$\Gamma_x^{st}(\lambda)^\mathcal{F}_{s}$ is strongly compact. For
this we only show that
$\Gamma_x^{st}(\lambda)^\mathcal{F}_{s}\subset C$. If possible,
let $\eta\in \Gamma_x^{st}(\lambda)^\mathcal{F}_{s}\setminus C$.
As $C$ is strongly compact so there is a $q>0$ such that
$\mathcal{N}_\eta(q)\cap C=\emptyset$. So we get
$\{k\in\mathbb{N}:\mathcal{F}_{x_k\eta}(q)>1-q\}\subset\{k\in\mathbb{N}:x_k\notin
C\}$ which implies that
$d_\lambda(\{k\in\mathbb{N}:\mathcal{F}_{x_k\eta}(q)>1-q\})=0$,
which contradicts that
$\eta\in\Gamma_x^{st}(\lambda)^\mathcal{F}_{s}$. So,
$\Gamma_x^{st}(\lambda)^\mathcal{F}_{s}\subset C$.
\end{proof}
\noindent\textbf{Acknowledgment:} The second author is grateful to
Council of Scientific and Industrial Research, India for his
fellowships funding under CSIR-JRF scheme during the preparation
of this paper.
\\
%--------------------------------------------------REFERENCES------------------------------

\end{document}